\documentclass[letterpaper,12pt,reqno,twoside]{amsart}

\usepackage{amssymb}
\usepackage{amsmath}
\usepackage[dvips]{graphicx}
\usepackage{psfrag,color}
\usepackage[colorlinks]{hyperref}

\newcommand{\be}{\begin{eqnarray}}
\newcommand{\ee}{\end{eqnarray}}

\newcommand{\beq}{\begin{equation}}
\newcommand{\eeq}{\end{equation}}
\newcommand{\beqn}{\begin{equation*}}
\newcommand{\eeqn}{\end{equation*}}

\newcommand{\defas}{\mathrel{\raise.095ex\hbox{$:$}\mkern-4.2mu=}}
\newcommand{\defasr}{\mathrel{=\mkern-4.2mu\raise.095ex\hbox{$:$}}}

\newcommand{\norm}[1]{\lVert#1\rVert}

\newcommand{\average}[1]{\langle#1\rangle}

\newcommand{\ie}{\textit{i.e.}}

\DeclareMathAlphabet{\mathfat}{U}{bbold}{m}{n}          % Identity operator; requires amsfonts

\newtheorem{thm}{Theorem}

\newtheorem{cor}[thm]{Corollary}
\newtheorem{lem}[thm]{Lemma}

\newtheorem{remark}[thm]{Remark}

\newcommand\cC{{\mathcal C}}

\newcommand\cF{{\mathcal F}}

\newcommand\cH{{\mathcal H}}

\newcommand\cL{{\mathcal L}}
\newcommand\cM{{\mathcal M}}

\newcommand\cO{{\mathcal O}}

\newcommand\cS{{\mathcal S}}

\newcommand\bC{{\mathbb C}}

\newcommand\bR{{\mathbb R}}
\newcommand\bS{{\mathbb S}}

\newcommand{\Var}{\operatorname{Var}}

\begin{document}

\title[A strong pair correlations bound implies the CLT]{A strong pair correlation bound implies the CLT for Sinai Billiards}

\author[Mikko Stenlund]{Mikko Stenlund}
\email{mikko@cims.nyu.edu}
\address[Mikko Stenlund]{
Courant Institute of Mathematical Sciences\\
New York, NY 10012, USA; Department of Mathematics and Statistics, P.O. Box 68, Fin-00014 University of Helsinki, Finland.}
\urladdr{http://www.math.helsinki.fi/mathphys/mikko.html}

\keywords{Sinai Billiards, Lorentz Gas, Anosov Diffeomorphisms, Central Limit Theorem}
\subjclass[2000]{60F05; 37D50, 37D20, 82D05}

\date{June 10, 2009}

\begin{abstract}
For Dynamical Systems, a strong bound on multiple correlations implies the Central Limit Theorem (CLT) \cite{ChernovMarkarian}. In Chernov's paper \cite{Chernov-BilliardsCoupling}, such a bound is derived for dynamically H\"older continuous observables of dispersing Billiards. Here we weaken the regularity assumption and subsequently show that the bound on multiple correlations follows directly from the bound on pair correlations. Thus, a strong bound on pair correlations alone implies the CLT, for a wider class of observables. The result is extended to Anosov diffeomorphisms in any dimension. 
\end{abstract}

\maketitle

%%%%%%%%%%%%%%%%%%%%%%%%%%%%
%%%%%%%%%%%%%%%%%%%%%%%%%%%%

\subsection*{Acknowledgement}
The author is partially supported by a fellowship from the Academy of Finland. He is grateful to Lai-Sang Young for her encouragement.

%%%%%%%%%%%%%%%%%%%%%%%%%%%%
%%%%%%%%%%%%%%%%%%%%%%%%%%%%

\section{Introduction}
In the context of dynamical systems, the question whether fast decay of pair correlations implies the Central Limit Theorem (CLT) is an open one. Hopes that the answer is positive are not completely without warrant: we have been unable to locate examples in the literature in which pair correlations decay rapidly but the CLT fails. On the other hand, cases in which the CLT has been inferred from little more than a pair correlation bound are few. The purpose of this paper is to investigate the possibility of using information solely regarding pair correlations for claiming control over multiple correlations, and further for obtaining the CLT. In a physically interesting example, we show that a single pair correlation bound implies the CLT for Sinai Billiards in a simple fashion. We also prove a similar result for Anosov diffeomorphisms in arbitrary dimension. 

Let $\cF:\cM\to\cM$ be a dynamical system with an invariant measure $\mu$. Measurable functions $f:\cM\to\bC$ are called observables and their averages are denoted $\average{f}=\int_\cM f\,d\mu$. Recall that a dynamical system is mixing if $\lim_{n\to\infty} \average{f \cdot  g\circ \cF^{n}} =  \average{f}\average{g}$ for all $f,g\in L^2(\mu)$.

For a mixing system, a typical bound on pair correlations is
\beq\label{eq:pair_corr_intro}
| \average{f \cdot  g\circ \cF^{n}} - \average{f}\average{g} | \leq c_{f,g}r_{f,g}(n)
\eeq
for all $n$, given that $f$ and $g$ belong to some suitable classes of observables, $\cH_1$ and $\cH_2$, respectively. Here the constant $c_{f,g}$ usually depends on a few properties of $f$ and $g$ --- such as their norms --- rather than the details of the functions. Moreover, as $n\to \infty$, $r_{f,g}(n)\to 0$ at a rate that only depends on a few properties of $f$ and $g$. For instance, this rate could be exponential and the exponent could depend on the regularity --- say the H\"older exponents --- of the observables. Then, $c_{f,g}$ and $r_{f,g}$ are uniform on proper subclasses of $\cH_i$.

Assume now that $f$ is real-valued and consider sums of the form
\beqn%\label{eq:sum}
S_n = \sum_{j=0}^{n-1} f\circ \cF^j.
\eeqn
Recall that the Central Limit Theorem (CLT) states that the distribution of the normalized sequence $S_n/{\sqrt{\Var S_n}}$ tends to the standard Gaussian distribution:
\beq\label{eq:CLT}
\lim_{n\to\infty} \mu\!\left( \frac{S_n}{\sqrt{\Var S_n}} \leq t\right) = \frac{1}{\sqrt {2\pi}}\int_{-\infty}^t e^{-s^2/2}\,ds \qquad \forall\,t\in\bR.
\eeq
If the sequence of auto correlations $C_f(n)=\average{f\cdot f\circ \cF^n}-\average{f}^2$ has a finite first moment, \ie,
\beq\label{eq:corr_sum}
\sum_{n=1}^\infty n|C_f(n)| < \infty,
\eeq
then a direct computation reveals that
\beq\label{eq:CLT_var}
\lim_{n\to\infty}\frac{\Var S_n}{n} = \sigma_f^2 = C_f(0) + 2\sum_{i=1}^{\infty}C_f(i),
\eeq
in which case $\sqrt{\Var S_n}$ can be replaced by $\sqrt{n}\sigma_f$ in \eqref{eq:CLT}. The summability of the auto correlations, at the very least, is necessary for the CLT to hold in this form. Establishing \eqref{eq:pair_corr_intro} has been an essential, but not sufficient, part of CLT proofs in the literature. 

In \cite{ChernovMarkarian}, it is shown that the CLT is actually implied by good decay of \emph{multiple} correlations. The authors give the Sinai Billiards as an application, and deduce the CLT from the strong bound on multiple correlations obtained in \cite{Chernov-BilliardsCoupling}. 

Multiple correlations can generically be viewed as pair correlations: assuming $f_1,\dots,f_k$ are observables and $n\geq 0$ and $0\leq i_1<\dots<i_l<i_{l+1}<\dots < i_k$ are integers, we have
\beqn
\average{f_1\circ \cF^{i_1}\cdots f_l\circ \cF^{i_l} \cdot f_{l+1}\circ \cF^{i_{l+1}+n}\cdots f_k\circ \cF^{i_k+n}} = \average{\tilde f \cdot \tilde g\circ \cF^{n}},
\eeqn
if we define
\beqn
\tilde f = f_1\circ \cF^{i_1}\cdots f_l\circ \cF^{i_l} \quad\text{and}\quad
\tilde g = f_{l+1}\circ \cF^{i_{l+1}}\cdots f_k\circ \cF^{i_k}.
\eeqn
This representation, the choice of $l$ and $n$, is rather arbitrary, but if the time gap $n$ can be chosen large, one may hope to benefit from \eqref{eq:pair_corr_intro} and claim control over multiple correlations in order to verify that the CLT holds. However, even if one is able to prove that $\tilde f\in \cH_1$ and $\tilde g\in \cH_2$, which is not always the case, it may turn out that $c_{\tilde f,\tilde g}$ or $r_{\tilde f,\tilde g}(n)$ is substantially worse than needed for the argument to work. In the following sections, the idea of controlling multiple correlations through pair correlations is developed further. In particular, applications to Billiards and Anosov diffeomorphisms are presented.

Unified approaches for obtaining the CLT for dynamical systems are available in the literature \cite{Chernov-MarkovApprox,Liverani-CLT}. They are general, but of quite different flavor compared to this work. The paper \cite{Chernov-MarkovApprox} has been formulated in terms of mixing properties of partitions of the phase space and is based on Markov approximations. Martingale approximations are used in \cite{Liverani-CLT}. In the invertible case, the latter too requires refined information about partitions, their measurability, existence of conditional measures, and estimates on the size of their elements. See, however, the first paragraph of Section~\ref{sec:pw_exp} discussing the non-invertible case. It is clear that, for the relevant parts, equivalent information must be contained in any CLT proof; in our examples it has surprisingly been encoded into a strong pair correlation bound.

\section{Obtaining the CLT from the decay of multiple correlations}
Here we make precise  the informal statement that sufficient control over \emph{multiple} correlations implies the CLT.  Theorem~\ref{thm:multicorr_CLT} is obtained by rephrasing parts of the proof of Theorem~7.19 in \cite{ChernovMarkarian} to suit our needs. As all the necessary details can be found in \cite{ChernovMarkarian}, we omit them here. We nevertheless mention that the backbone is a clever trick due to Bernstein \cite{Bernstein1927, Bernstein} for approximating sums of weakly dependent random variables by those of independent ones. The CLT is then inferred by checking that a Lindeberg condition on the independent variables is satisfied.

To formulate the theorem, we first partition the time interval $[0,n-1]$ into an alternating sequence of long and short intervals. The long intervals are of length\footnote{Here $[x]$ is the integer part of a number $x$.} $p=[n^a]$ and the short ones are of length $q=[n^b]$ with $0<b<a<\frac 12$. Hence, there are precisely $k=[n/(p+q)]\sim n^{1-a}$ pairs of long and short intervals (and a remaining interval of length $n-k(p+q)<p+q$).
\begin{thm}\label{thm:multicorr_CLT}
Denote $g=e^{itf/\sqrt{k\Var S_p}}$ and, for each $1\leq r\leq k$, 
\beqn
w_r = \left(g\cdot g\circ\cF \cdots g\circ \cF^{p-1}\right) \circ \cF^{(p+q)(r-1)} = w_1 \circ \cF^{(p+q)(r-1)}.
\eeqn 
If \eqref{eq:corr_sum} and
\beq\label{eq:CF_multi}
\lim_{n\to\infty} | \average{w_1\cdots w_k} - \average{w_1}\cdots\average{w_k} | = 0 \qquad \forall\,t\in\bR
\eeq
hold, then the CLT \eqref{eq:CLT} is satisfied together with \eqref{eq:CLT_var}.
\end{thm}
\begin{remark}
We stress that $p$, $q$, $k$, $g$, and $w_r$ all depend on $n$. Notice that a time gap of length $q$ separates the variables $w_r$ which are `supported' on the long intervals of length $p$.
\end{remark}

To shed a bit of light on the method of \cite{ChernovMarkarian}, let us denote by $\Delta_r$, $1\leq r\leq k$, the long intervals and split the sum
\beqn
S_n = S_n' + S_n'',
\eeqn
where $S_n'$ and $S_n''$ are the sum over all the long intervals, $\cup_r\Delta_r$, and the remainder, $[0,n-1]\setminus \cup_r\Delta_r$, respectively:
\beqn
S_n' = \sum_{r=1}^k S_p^{(r)} \quad\text{with}\quad S_p^{(r)} = \sum_{i\in\Delta_r} f\circ\cF^i.
\eeqn
Notice that the number of long and short intervals as well as the length of each interval increases as $n$ increases. However, the fraction of the entire time interval $[0,n-1]$ covered by $\cup_r\Delta_r$ tends to 1, because 
\beqn
\lim_{n\to\infty }\frac{p+q}{p} = \lim_{n\to\infty}\frac{[n^a]+[n^b]}{[n^a]} = 1.
\eeqn
Therefore, the sum $S_n''$ can asymptotically be neglected. Moreover, as the variables $f\circ\cF^i$ are weakly dependent and as the gaps between the long intervals $\Delta_r$ increase with $n$, the sums $S_p^{(r)}$ become asymptotically independent. This way, $S_n$ can be approximated by a sum of i.i.d.\@ random variables and the asymptotic normality of $S_n/\sqrt{n}$ be verified.

It appears that \cite{Chernov-BilliardsCoupling,ChernovMarkarian} are the first places where the possibility of obtaining the CLT based solely on a strong bound on multiple correlations has been explicitly mentioned, although Bernstein's trick has been introduced into the study of dynamical systems at least as early as \cite{BunimovichSinai-BM}. We would like to stress that it is not the rate of decay per se, but the \emph{form} of the bound on multiple correlations that counts in the proof appearing in \cite{ChernovMarkarian}; hence our choice to call such bounds strong as opposed to, say, exponential. See Corollary~\ref{cor:multicorr_bound} and Remark~\ref{rem:multicorr_bound}. Indeed, as explained in the Introduction, fast decay of pair correlations could yield fast decaying bounds on multiple correlations which, however, are too weak to be used in conjunction with Theorem~\ref{thm:multicorr_CLT} for deducing the CLT.

We remind the reader once more that each $w_r$ in \eqref{eq:CF_multi} depends on $n$.
Nevertheless, we mimic pair correlation notation by rewriting the multiple correlations in \eqref{eq:CF_multi} as the telescoping sums
\beqn%\label{eq:telescope}
%\begin{split}
 \average{w_1\cdots w_k} - \average{w_1}\cdots\average{w_k}
 = \sum_{r=1}^{k-1}\average{w_1}\cdots\average{w_{r-1}} \left[\average{w_r\cdots w_k}-\average{w_r}\average{w_{r+1}\cdots w_k}\right].
%\end{split}
\eeqn
Recalling $|\average{w_r}| \leq 1$ and using stationarity, we obtain
\begin{cor}\label{cor:corr_sum_CLT}
Denote $W_r = w_1\cdots w_{r-1}$. If \eqref{eq:corr_sum} and
\beqn%\label{eq:multi_pair}
\lim_{n\to\infty} \sum_{r=2}^{k} \left | \average{w_1 \cdot W_r\circ \cF^{p+q}}-\average{w_1}\average{W_r} \right |=0  \qquad \forall\,t\in\bR
\eeqn
hold, then the CLT \eqref{eq:CLT} is satisfied together with \eqref{eq:CLT_var}.
\end{cor}
%Despite its suggestive appearance, \eqref{eq:multi_pair} is not a condition on pair correlations in a conventional sense; $w_1$ and $W_r$ depend heavily on $n$.

A message of the present paper is that, while fast decay of pair correlations alone may not suffice for the CLT to hold, \emph{detailed information about their structure sometimes will}. Below we will show that two interesting classes of dynamical systems, Sinai Billiards (see \cite{ChernovDolgopyatICM} for applications) and Anosov diffeomorphisms, actually both admit a strong pair correlation bound that yields the CLT directly by supplying strong enough bounds on multiple correlations. As far as the author knows, these are the first examples of the kind.

%%%%%%%%%%%%%%%%%%%%%%%%%%%%%%%%%%%%%%%%%
%%%%%%%%%%%%%%%%%%%%%%%%%%%%%%%%%%%%%%%%%

\section{Sinai Billiards}\label{sec:billiards}
Here Sinai Billiards \cite{Sinai70} refers to the 2D periodic Lorentz gas with dispersing scatterers and finite horizon (finite free path). We will only list some facts about such systems, but the reader unfamiliar with Billiards should be able to follow the reasoning by taking the estimates in this section for granted. For background, see \cite{ChernovMarkarian,Szasz,Tabachnikov}. 

Recall that the dynamics of Sinai Billiards induces a billiard map $\cF:\cM\to\cM$ on the collision space $\cM$, which preserves a smooth ergodic SRB measure $\mu$. This map is uniformly hyperbolic but has a set of singularities due to tangential collisions. In a standard representation of the collision space, tangential collisions correspond to horizontal lines --- the boundary $\cS_0$ of $\cM$. The map also suffers of unbounded distortion because of the same reason. In order to deal with this nuisance, the space $\cM$ is divided by countably many horizontal lines into a disjoint union of strips \cite{BunimovichSinaiChernov91, Chernov-BilliardsCoupling}, on each of which distortions can be controlled. Let us denote the union of such lines $\bS$. This way the space is divided  by $\cS_0\cup\bS$ into countably many connected components which we call homogeneity strips or briefly H-strips.

There are two special families of cones associated with billiards. The unstable cones are invariant under $\cF$ and the stable ones under $\cF^{-1}$. A smooth curve is called stable if at every point its tangent vector belongs to the stable cone. The H-strips divide any stable curve into disjoint \emph{H-components}. A stable curve $W$ is a stable manifold if $\cF^n W$ is a stable curve for all $n>0$. Notice that the image $\cF^n W$ of a stable manifold $W$ shrinks as $n$ increases, but may well consist of several H-components.  A stable manifold $W$ is a homogeneous stable manifold, if $\cF^n W$ has just one H-component for all $n>0$. Unstable curves, unstable manifolds, and homogeneous unstable manifolds are defined analogously by considering unstable cones and backward iterates of $\cF$. 

For all $x,y$, we define the \emph{future separation time}
\beqn
s_+(x,y) = \min\{n\geq 0\,:\, \text{$\cF^nx$ and $\cF^ny$ lie in different H-strips}\}
\eeqn
and the \emph{past separation time}
\beqn
s_-(x,y) = \min\{n\geq 0\,:\, \text{$\cF^{-n}x$ and $\cF^{-n}y$ lie in different H-strips}\}.
\eeqn

We will now introduce two notions of regularity.
We say $f$ is \emph{dynamically H\"older continuous on homogeneous unstable manifolds} and write $f \in \cH^+_\star$, if there exist $K_f\geq 0$ and $\vartheta_f \in (0,1)$ such that, for any homogeneous unstable manifold $W^u$,
\beq\label{eq:dyn_Holder_+}
|f(x)-f(y)| \leq K_f \vartheta_f^{s_+(x,y)}\qquad \forall\, x,y\in W^u.
\eeq
Similarly, we say $f$ is \emph{dynamically H\"older continuous on homogeneous stable manifolds} and write $f\in \cH^-_\star$, if there exist $K_f\geq 0$ and $\vartheta_f \in (0,1)$ such that, for any homogeneous stable manifold $W^s$,
\beq\label{eq:dyn_Holder_-}
|f(x)-f(y)| \leq K_f \vartheta_f^{s_-(x,y)}\qquad \forall\, x,y\in W^s.
\eeq
To make comparisons with \cite{Chernov-BilliardsCoupling} in the following, we denote $f\in \cH^+$ if \eqref{eq:dyn_Holder_+} holds on all unstable \emph{curves} $W^u$ and $f\in \cH^-$ if \eqref{eq:dyn_Holder_-} holds on all stable \emph{curves} $W^s$.

Obviously, $\cH_\star^\pm$ are vector spaces. If $f,g\in\cH_\star^\pm$ are bounded, then $fg\in\cH_\star^\pm$ with
\beq\label{eq:prod_Holder}
K_{fg} = {\norm{f}}_\infty K_g + K_f {\norm{g}}_\infty\quad\text{and}\quad \vartheta_{fg} = \max(\vartheta_f,\vartheta_g).
\eeq
The classes $\cH^{\pm}_\star$ enjoy the following stability property under the action of $\cF$:
\begin{lem}\label{lem:stability}
If $f\in\cH^-_\star$, then $f\circ \cF\in\cH^-_\star$ with
\beqn
\vartheta_{f\circ \cF} = \vartheta_f \quad\text{and}\quad K_{f\circ \cF} = K_f\vartheta_f.
\eeqn
If $f\in\cH^+_\star$, then $f\circ \cF^{-1}\in\cH^+_\star$ with
\beqn
\vartheta_{f\circ \cF^{-1}} = \vartheta_f \quad\text{and}\quad K_{f\circ \cF^{-1}} = K_f\vartheta_f.
\eeqn
\end{lem}
\begin{proof}
Let $f\in\cH^-_\star$. Assume that $W^s$ is a homogeneous stable manifold. If $x,y\in W^s$, then $\cF x,\cF y\in \cF W^s$ which is also a homogeneous stable manifold. Hence, $s_-(\cF x,\cF y) = 1+s_-(x,y)$ and
\beqn
|f(\cF x)-f(\cF y)|\leq K_f \vartheta_f^{s_-(\cF x,\cF y)} = (K_f\vartheta_f)\vartheta_f^{s_-(x, y)}.
\eeqn
The case $f\in\cH^+_\star$ follows by reversing time.
\end{proof}
\begin{remark}
The proof of the $\cH^-_\star$ part of Lemma~\ref{lem:stability} relies on the fact that regularity is only required on \emph{homogeneous stable manifolds}. In \cite{Chernov-BilliardsCoupling}, regularity on all stable \emph{curves} is required, which results in the class $\cH^-$ of more regular functions. However, the $\cF$-image of a stable curve is not necessarily a stable curve but rather tends to align with the unstable direction. One may then consider relaxing the regularity condition and restricting to  stable \emph{manifolds} or H-components of stable manifolds\footnote{Such a case is recovered by collapsing the invariant cones into invariant lines \cite{Chernov-BilliardsCoupling}.}. However, while the $\cF$-image $\cF W$ of such an H-component $W$ is stable, it may have several H-components, in which case $s_-(\cF x,\cF y)=0$ for some $x,y\in W$. The remedy is to give up more regularity and to consider the class $\cH^-_\star$, \ie, observables that are dynamically H\"older continuous on homogeneous stable manifolds.

Similar remarks apply for the $\cH^+_\star$ part.
\end{remark}

\begin{cor}\label{cor:dyn_holder_invariant}
Suppose $\tilde f = f_0\circ \cF^{i_0} \cdots f_k\circ \cF^{i_k}$, where $0\leq i_0<\dots<i_k$ and each $f_i\in\cH^-_\star$ is bounded. Setting $K_{\{f_i\}} = \max_i K_{f_i}$ and $\vartheta_{\{f_i\}} = \max_i\vartheta_{f_i}$, we have $\tilde f\in \cH^-_\star$ with
\beqn
\vartheta_{\tilde f} = \vartheta_{\{f_i\}} \quad\text{and}\quad K_{\tilde f}= K_{\{f_i\}} \frac{\prod_{i} \norm{f_i}_\infty } {\min_i \norm{f_i}_\infty } \frac{\vartheta_{\{f_i\}}^{i_0}}{1-\vartheta_{\{f_i\}}}.
\eeqn
Similarly, if each $f_i\in\cH^+_\star$ is bounded, then $\tilde f = f_0\circ \cF^{-i_0} \cdots f_k\circ \cF^{-i_k}\in\cH^+_\star$ with $\vartheta_{\tilde f}$ and $K_{\tilde f}$ as above.
\end{cor}
\begin{proof}
Follows by induction from \eqref{eq:prod_Holder} and Lemma~\ref{lem:stability}.
\end{proof}

Here we present a strengthened version of Theorem 4.3 of \cite{Chernov-BilliardsCoupling}. $\vartheta_\Upsilon<1$, $\kappa>0$, and $C_0> 0$ are constants whose definitions can be found in that paper. The proof is at the end of the section.
\begin{thm}\label{thm:chernov_pair_correlation}
For every bounded pair $f\in\cH^+_\star$, $g\in\cH^-_\star$, and $n\geq 0$,
\beqn
| \average{f \cdot  g\circ \cF^{n}} - \average{f}\average{g} | \leq B_{f,g}\theta_{f,g}^n,
\eeqn
where $\average{\,\cdot\,}$ denotes the $\mu$-integral,
\beq\label{eq:theta_def}
\theta_{f,g} = \left[ \max\{\vartheta_\Upsilon,\vartheta_f,\vartheta_g,e^{-1/\kappa}\} \right]^{1/4} < 1,
\eeq
and
\beqn
B_{f,g} = C_0(K_f\norm{g}_\infty+\norm{f}_\infty K_g+\norm{f}_\infty\norm{g}_\infty).
\eeqn
\end{thm}
In \cite{Chernov-BilliardsCoupling} the formulation of the theorem requires $f,g\in\cH^-\cap\cH^+$ but it is remarked that the proof requires only the weaker assumption $f\in\cH^+$ and $g\in\cH^-$. We have further relaxed the assumptions. This is an important point to us, as will now become apparent.
Owing to the stability properties of the classes $\cH^-_\star$ and $\cH^+_\star$ stated in Corollary~\ref{cor:dyn_holder_invariant}, Theorem~\ref{thm:chernov_pair_correlation} immediately implies the following bound on \emph{multiple correlations}. Generalization to nonconsecutive times is easy.
\begin{cor}\label{cor:multicorr_bound}
Let $\tilde f = f_0 \cdot f_1\circ\cF^{-1} \cdots f_r\circ \cF^{-r}$ and $\tilde g = g_0\cdot g_1\circ \cF^{1} \cdots g_k\circ \cF^{k}$, where $f_0,\dots, f_r\in \cH^+_\star$ and $g_0,\dots,g_k\in \cH^-_\star$ are bounded and have identical parameters\footnote{This can always be arranged by scaling the functions and choosing the weakest parameters. It is also a simple task to modify the statement so as to remove this condition. This does not serve our purpose here.} in the sense that $\vartheta_{f_i}=\vartheta_{f_0}$, $K_{f_i}=K_{f_0}$, $\norm{f_i}_\infty = \norm{f_0}_\infty$, $\vartheta_{g_i}=\vartheta_{g_0}$, $K_{g_i}=K_{g_0}$, and $\norm{g_i}_\infty = \norm{g_0}_\infty$. Then
\beqn
\left | \average{\tilde f \cdot  \tilde g\circ \cF^{n}} - \average{\tilde f}\average{\tilde g} \right | \leq B_{\tilde f,\tilde g}\theta_{f_0,g_0}^n
\eeqn
for all $n\geq 0$, where $\theta_{f_0,g_0}$ is as in \eqref{eq:theta_def} and
\beqn
B_{\tilde f, \tilde g} = C_0\norm{f_0}_\infty^r \norm{g_0}_\infty^k \left(\frac{K_{f_0}}{1-\vartheta_{f_0}}\norm{g_0}_\infty+\norm{f_0}_\infty \frac{K_{g_0}}{1-\vartheta_{g_0}}+\norm{f_0}_\infty\norm{g_0}_\infty\right).
\eeqn
\end{cor}
\begin{remark}\label{rem:multicorr_bound}
The reader should pay attention to the form of the prefactor $B_{\tilde f,\tilde g}$. Any growth with $r$ and $k$ is associated with the norms $\norm{\,\cdot\,}_{\infty}$, which measure \emph{size}, not with the dynamical H\"older constants $K_{\cdot}$, which measure \emph{regularity}. Moreover, the rate of decay remains under control (in fact unchanged) as $r$ and $k$ increase.
\end{remark}
Corollary~\ref{cor:multicorr_bound} strengthens Theorem~4.5 of \cite{Chernov-BilliardsCoupling} regarding the regularity assumption.
We are in position to prove
\begin{thm}\label{Billiards_CLT}
If $f\in \cH^-_\star\cap\cH^+_\star$ is real-valued and bounded, then the Central Limit Theorem \eqref{eq:CLT} holds together with \eqref{eq:CLT_var}.
\end{thm}
\begin{proof}
Because the auto correlations $C_f(n)$ decay sufficiently fast by Theorem~\ref{thm:chernov_pair_correlation}, condition \eqref{eq:corr_sum} is satisfied which implies \eqref{eq:CLT_var}. According to Corollary~\ref{cor:corr_sum_CLT}, it suffices to bound $\left | \average{w_1 \cdot W_r\circ \cF^{p+q}}-\average{w_1}\average{W_r} \right |$ (see the definitions of $w_1$ and $W_r$ in Theorem~\ref{thm:multicorr_CLT} and Corollary~\ref{cor:corr_sum_CLT}). This involves functions of the form $g=e^{itf/\sqrt{k\Var S_p}}\in  \cH^-_\star\cap\cH^+_\star$,  for which ${\norm{g}}_\infty=1$, $\vartheta_g=\vartheta_f$, and $K_g\leq (|t|/\sqrt{k\Var S_p}) K_f =\cO(1/\sqrt{n}) |t| K_f$. We then use Corollary~\ref{cor:multicorr_bound} with $w_1\circ\cF^{-p}$ assuming the role of $\tilde f$ and $W_r$ that of $\tilde g$ (recall that $\cF$ is invertible):
\beqn
\begin{split}
 \left | \average{w_1 \cdot W_r\circ \cF^{p+q}}-\average{w_1}\average{W_r} \right |
% & = \left | \average{w_1\circ\cF^{-p} \cdot W_r\circ \cF^{q}}-\average{w_1\circ\cF^{-p}}\average{W_r} \right | \\
& \leq C_0\left(\frac{\cO(1/\sqrt{n}) |t| K_f}{1-\vartheta_f}+1\right) \theta_{f,f}^q.
\end{split}
\eeqn
Finally, recall from above Theorem~\ref{thm:multicorr_CLT} that the quantity $k$ appearing in Corollary~\ref{cor:corr_sum_CLT} grows as $n^{1-a}$. Hence, $\lim_{n\to 0} k\theta_{f,f}^q = 0$, and the CLT follows by Corollary~\ref{cor:corr_sum_CLT}.
\end{proof}

The class $\cH^-_\star\cap\cH^+_\star$ of observables $f$ is wider than the class $\cH^-\cap\cH^+$ of \cite{Chernov-BilliardsCoupling}. Notice in particular that we have derived the CLT directly from a pure pair correlation bound, Theorem~\ref{thm:chernov_pair_correlation}.

In  \cite{Chernov-BilliardsCoupling}, the proof of Theorem~4.5 (which corresponds to Corollary~\ref{cor:multicorr_bound} above) is almost identical with the proof of Theorem~4.3 (which corresponds to Theorem~\ref{thm:chernov_pair_correlation} above). We have seen above that Theorem~4.5 actually becomes a direct consequence of Theorem~4.3 after the classes $\cH^+_\star$ and $\cH^-_\star$ have been incorporated as in Theorem~\ref{thm:chernov_pair_correlation} and Corollary~\ref{cor:multicorr_bound}. 

\begin{remark} 
In \cite{Chernov-BilliardsCoupling}, Theorem~4.3 is derived from an ``equidistribution property'' stated in Theorem~4.1. Similarly, Theorem~4.5 is derived from Theorem~4.2 which states an ``equidistribution property'' for multiple observation times. Both Theorem~4.1 and 4.2 require the observables to be in $\cH^-\cap\cH^+$, but as is pointed out in the paper (and clear from the proof), they hold under the weaker assumption that the observables be in $\cH^-$. Interestingly, Theorem~4.1 can be strengthened so as to hold for $\cH^-_\star$. After this it implies directly a stronger form of Theorem~4.2 that holds for $\cH^-_\star$.
\end{remark}

Next, we give the argument leading to Theorem~\ref{thm:chernov_pair_correlation}. Instead of repeating large parts of \cite{Chernov-BilliardsCoupling} in which all the basic work has been done, we point directly to the places in it where care is needed.
\begin{proof}[Proof of Theorem \ref{thm:chernov_pair_correlation}]
The Coupling Lemma \cite[Lemma~3.4]{Chernov-BilliardsCoupling} does not concern observables at all and is thus immune to relaxing their regularity. The Equidistribution Theorem \cite[Theorem~4.1]{Chernov-BilliardsCoupling} holds assuming just $g\in\cH^-_\star$ \footnote{In the original text, the observable is denoted $f$. We have renamed it $g$ not to create confusion in the rest of this proof.}. This is so, because regularity is only used in the estimate (4.4) of the proof, and this estimate remains valid if $g\in\cH^-_\star$ as the coupled points $x$ and $y$ lie on the same \emph{homogeneous stable manifold}\footnote{In \cite{Chernov-BilliardsCoupling}, homogeneous (un)stable manifolds are often called (un)stable H-manifolds.} according to the Coupling Lemma. 

One then has to verify that relaxing the regularity condition for the pair correlation bound of \cite[Theorem~4.3]{Chernov-BilliardsCoupling} is legitimate. The bound on the quantity $\delta_\alpha$ appearing below (4.11) in \cite{Chernov-BilliardsCoupling} is clearly the only place where the regularity of $g$ matters. This bound remains unaffected if we assume $g\in\cH^-_\star$, since it only relies on the bound in Theorem~4.1 which remains unchanged as we saw above. As far as the regularity of $f$ is concerned, only homogeneous unstable manifolds count. This is obvious from the definition of $\bar f$ and the inequality (4.10). Thus, it is enough to take $f\in\cH^+_\star$.
\end{proof}

%%%%%%%%%%%%%%%%%%%%%%%%%%%%%%%%%%%%%%%%%
%%%%%%%%%%%%%%%%%%%%%%%%%%%%%%%%%%%%%%%%%

\section{Anosov diffeomorphisms}
We will exemplify with Anosov diffeomorphisms the passage from a strong pair correlation bound to the CLT. As Anosov diffeomorphisms lack the singularities of Billiards, one rightly expects everything to work as in the previous section.  Notice, however, that throughout Section~\ref{sec:billiards} it was assumed that the space is 2-dimensional. We include this section because of its transparency and because detailed pair correlation bounds are available in any dimension. 

In the following, $C^{1+\alpha}$ stands for differentiable functions whose first derivative is $C^\alpha$, \ie, H\"older continuous with exponent $0<\alpha<1$. Let $\cM$ be a $d$-dimensional Riemannian manifold and $\cF$ a transitive $C^{1+\alpha}$ Anosov diffeomorphism on it. Let $d^s(x,y)$ denote the distance between $x$ and $y$ along a stable manifold ($=\infty$ if $x$ and $y$ are not on the same stable manifold, \ie, $y\notin W^s(x)$). Similarly, let $d^u(x,y)$ be the distance  along an unstable manifold. There exists $0<\nu<1$ such that $d^s(\cF x,\cF y)\leq \nu d^s(x,y)$ if $x\in W^s(y)$ and $d^u(\cF^{-1} x,\cF^{-1} y)\leq \nu d^u(x,y)$ if $x\in W^u(y)$.

We recall some definitions from \cite{BressaudLiverani}. Fix $\delta>0$ and $0<\beta<1$. Define, for all $f:\cM\to\bC$,
\beqn
{|f|}_s = \sup_{d^s(x,y)\leq \delta}\frac{|f(x)-f(y)|}{d^s(x,y)^\beta}, \qquad {\norm{f}}_s = {\norm{f}}_\infty + {|f|}_s,
\eeqn
and
\beqn
{|f|}_u = \sup_{d^u(x,y)\leq \delta}\frac{|f(x)-f(y)|}{d^u(x,y)^\alpha}, \qquad {\norm{f}}_u = {\norm{f}}_1 + {|f|}_u,
\eeqn
where the $L^1$-norm is defined with respect to the Riemannian volume. Finally, $\cC_s$ stands for the set of Borel measurable functions $f:\cM\to\bC$ with ${\norm{f}}_s<\infty$.
 
Although results similar to the one below have been known much earlier, we cite \cite[Corollary 2.1]{BressaudLiverani} due to the precise form of the bound there.
\begin{thm}\label{thm:Anosov_pair_correlation}
There exists a unique $\cF$-invariant SRB measure $\mu$. There exist $0<\theta<1$ and $C_0>0$ such that, for all $f\in C^\alpha$ and all $g\in \cC_s$,
\beqn
| \average{f \cdot  g\circ \cF^{n}} - \average{f}\average{g} | \leq C_0{\norm{f}}_u{\norm{g}}_s \theta^n,
\eeqn
where $\average{\,\cdot\,}$ stands for the $\mu$-integral.
\end{thm}

The following bound on multiple correlations is readily implied. Extending it to nonconsecutive times is easy.
\begin{cor}\label{cor:Anosov_multicorr_bound}
Let $\tilde f = f_0 \cdot f_1\circ\cF^{-1} \cdots f_r\circ \cF^{-r}$ and $\tilde g = g_0\cdot g_1\circ \cF^{1} \cdots g_k\circ \cF^{k}$, where $f_0,\dots, f_r\in C^\alpha$ and $g_0,\dots,g_k\in \cC_s$. Assume also that each ${\norm{f_i}}_\infty={\norm{f_0}}_\infty$ and each ${\norm{g_i}}_\infty={\norm{g_0}}_\infty$. Then,
\beqn
\left | \average{\tilde f \cdot  \tilde g\circ \cF^{n}} - \average{\tilde f}\average{\tilde g} \right | \leq B_{\tilde f,\tilde g}\theta^n,
\eeqn
where
\beqn
B_{\tilde f, \tilde g} = C {\norm{f_0}}_\infty^r {\norm{g_0}}_\infty^k \left(\max_i{| f_i |}_u{\norm{g_0}}_\infty+{\norm{f_0}}_\infty \max_i{| g_i |}_s+{\norm{f_0}}_\infty{\norm{g_0}}_\infty\right).
\eeqn
\end{cor}
\begin{proof}
Assuming $d^s(x,y)\leq \delta$, 
\beqn
\begin{split}
|\tilde g(x)-\tilde g(y)| & \leq \sum_{l=0}^k \left| g_0(x) \cdots g_{l-1}(\cF^{l-1}x) \right| \left| g_l(\cF^{l} x)-g_l(\cF^{l} y) \right|  \left| g_{l+1}(\cF^{l+1}y) \cdots g_k( \cF^{k}y) \right| \\
& \leq \sum_{l=0}^k\left(\prod_{i\neq l} \norm{g_i}_\infty \right) \! \left| g_l(\cF^{l} x)-g_l(\cF^{l} y) \right| \leq  \sum_{l=0}^k \left(\prod_{i\neq l} \norm{g_i}_\infty \right) \! {| g_l |}_s d^s(\cF^l x,\cF^l y)^\beta \\ 
& \leq \sum_{l=0}^k \left(\prod_{i\neq l} \norm{g_i}_\infty \right) \! {| g_l |}_s\nu^{\beta l} d^s(x,y)^\beta 
\leq \frac{\prod_{i} \norm{g_i}_\infty } {\min_i \norm{g_i}_\infty } \frac{\max_l{| g_l |}_s}{1-\nu^\beta}d^s(x,y)^\beta.
\end{split}
\eeqn
From this we obtain a bound on ${|f|}_s$ which implies
\beqn
{\norm{\tilde g}}_s \leq \frac{1}{1-\nu^\beta} \left(\prod_{i} \norm{g_i}_\infty\right) \left(1+\frac{\max_l{| g_l |}_s}{\min_i \norm{g_i}_\infty}\right)
\eeqn
and that $\tilde g\in\cC_s$.
It is clear that $\tilde f\in C^\alpha$. Mimicking the treatment of $\tilde g$ and using ${\norm{\tilde f}}_1 \leq {\norm{\tilde f}}_\infty {\norm{1}}_1$, we also get
\beqn
{\norm{\tilde f}}_u \leq \frac{\max(1,{\norm{1}}_1)}{1-\nu^\alpha} \left(\prod_{i} \norm{f_i}_\infty\right) \left(1+\frac{\max_l{| f_l |}_u}{\min_i \norm{f_i}_\infty}\right).
\eeqn
We can then apply Theorem~\ref{thm:Anosov_pair_correlation}.
\end{proof}
Notice that the proof contains a stability result similar to Lemma~\ref{lem:stability}. Following \cite{ChernovMarkarian}, the Central Limit Theorem can then be established immediately with the aid of Corollary~\ref{cor:Anosov_multicorr_bound}.
\begin{thm}
If $f\in C^\alpha\cap\cC_s$ is real-valued, then the Central Limit Theorem \eqref{eq:CLT} holds together with \eqref{eq:CLT_var}.
\end{thm}
\begin{proof}
The proof of Theorem~\ref{Billiards_CLT} applies, mutatis mutandis.
\end{proof}

Corollary~\ref{cor:Anosov_multicorr_bound} is interesting in its own right, as it gives a rather explicit bound on the multiple correlations for generic transitive Anosov diffeomorphisms.

%%%%%%%%%%%%%%%%%%%%%%%%%%%%%%%%%%%%%%%%%
%%%%%%%%%%%%%%%%%%%%%%%%%%%%%%%%%%%%%%%%%

\section{A non-invertible example}\label{sec:pw_exp}
Here we present an example of a non-invertible, piecewise expanding, dynamical system in which a known pair correlation bound (Theorem~\ref{thm:pw_corr}) can be used to prove the CLT. However, this time the situation is not as straightforward, and an extra ingredient, the Lasota--Yorke inequality appearing in \eqref{eq:pw_LY_Leb}, will be needed. The latter is known to imply a spectral gap and further the CLT. Nonetheless, our proof is formally independent of spectral arguments and we have chosen to include it as an interesting example of the usage of Corollary \ref{cor:corr_sum_CLT}.  In the case at issue, a notably direct way of passing from the pair correlation bound to the CLT has been established in \cite{Liverani-CLT}. It is based on a general result on martingale approximations rather than bounding multiple correlations, which seems particularly well suited to the non-invertible setting.

Let $\cF:[0,1]\to[0,1]$ be a piecewise $C^2$, uniformly expanding, map and suppose $d\mu=\phi dx$ is an absolutely continuous invariant measure with respect to the Lebesgue measure $dx$ with density $\phi\in L^1([0,1],dx)$. We write ${\norm{f}}_1 = \int_0^1|f|\,dx$. A complex-valued function $g$ defined on $[a,b]$ is of bounded variation, denoted $g\in BV[a,b]$, if the total variation $\bigvee_a^bg = \sup_{(x_i)}\sum_i |g(x_i)-g(x_{i+1})| $ is finite. Here the supremum runs over all finite partitions of $[a,b]$. The following theorem \cite{HofbauerKeller} can be found in \cite{Liverani-pwexp} up to trivial modifications.
\begin{thm}\label{thm:pw_corr}
Suppose the system $(\cF,\mu)$ is mixing and that $\inf\phi>0$. Then there exist constants $b>0$, $K>0$, and $\Lambda\in(0,1)$ such that, for each $f\in L^1([0,1],dx)$ and $g\in BV[0,1]$
\beq\label{eq:pw_corr}
\left| \int_0^1 f\circ \cF^n g\,dx - \int_0^1 f\,d\mu\int_0^1g\,dx \right| \leq K\Lambda^{-n} {\norm{f}}_1\left({\norm{g}}_1+b\bigvee_0^1 g \right).
\eeq
\end{thm}
\begin{remark}
Under the assumptions of the theorem, $\phi\in BV[0,1]$. 
\end{remark}

Let us denote $\lambda = \inf |\cF'|$ and assume $\lambda>2$, by considering a sufficiently large power of $\cF$ if necessary. Let $\cL$ stand for the transfer operator of $\cF$ with respect to the Lebesgue measure:
\beqn
(\cL g)(x) = \sum_{y\in \cF^{-1}x} \frac{g(y)}{|\cF'(y)|}.
\eeqn
There exists a constant $A$, depending on the map $\cF$, such that
\beq\label{eq:pw_LY_Leb}
\bigvee_0^1 ( \cL g) \leq  2\lambda^{-1}\bigvee_0^1 g + A{\norm{g}}_1
\eeq
holds for all $g\in BV[0,1]$ \cite{LasotaYorke}. Together, \eqref{eq:pw_corr} and \eqref{eq:pw_LY_Leb} imply the CLT in a straightforward fashion, as we will now see.

\begin{remark}
The Lasota--Yorke inequality \eqref{eq:pw_LY_Leb} is usually used with real-valued functions. From its proof in \cite{Liverani-pwexp} it is apparent that it holds true for complex-valued functions. Also the bound in \eqref{eq:pw_corr} extends to complex-valued functions. First assume $f=u+iv$ is complex and $g$ is real, use \eqref{eq:pw_corr}, and notice that ${\norm{u}}_1+{\norm{v}}_1\leq \sqrt{2}{\norm{f}}_1$. Next, assume also $g=t+iw$ is complex, use \eqref{eq:pw_corr}, and notice that $\bigvee_0^1 t+\bigvee_0^1 w\leq 2\bigvee_0^1 g$.
\end{remark}

Setting $W_r = w_1\cdots w_{r-1}$ and $\average{f} = \int_0^1 f\,d\mu$ as before,
\beq\label{eq:multi_transfer}
%\begin{split}
\average{w_1\cdots w_r}  = \average{w_1\cdot W_r\circ \cF^{p+q}}
%= \int_0^1w_1\cdot W_r\circ \cF^{p+q}\,\phi dx 
 =  \int_0^1 \cL^{p-1}(\phi w_1)\cdot W_r\circ \cF^{q+1} \,dx .
%\end{split}
\eeq
Moreover,
\beqn
\int_0^1 \cL^{p-1}(\phi w_1) \,dx = 
%\int_0^1 \phi w_1 \,dx =  
\average{w_1}
\quad\text{and}\quad
\int_0^1W_r\,d\mu = \average{w_2\cdots w_r},
\eeqn
where invariance has been used. Hence, \eqref{eq:pw_corr} yields
\beq\label{eq:expanding_multicorr}
\begin{split}
| \average{w_1\cdots w_r} - \average{w_1}\average{w_2\cdots w_r} | 
%& \leq K\Lambda^{-q-1} {\norm{W_r}}_1\left({\norm{\cL^{p-1}(\phi w_1)}}_1+b\bigvee_0^1 \cL^p(\phi w_1) \right) \\
& \leq K\Lambda^{-q-1} \left(1+b\bigvee_0^1 \cL^{p-1}(\phi w_1) \right),
\end{split}
\eeq
as $|w_i|=1$ implies ${\norm{W_r}}_1 = 1$ and ${\norm{\cL^{p-1}(\phi w_1)}}_1\leq {\norm{\phi w_1}}_1 = {\norm{\phi}}_1 = 1$.

We are done if we can establish a good bound on $\bigvee_0^1 \cL^{p-1}(\phi w_1)$. A straightforward iteration of \eqref{eq:pw_LY_Leb} will not suffice, because $\bigvee_0^1 (\phi w_1)$ seems to grow at a dominating exponential rate due to the presence of $\cF^{p-1}$ in $w_1$. Controlling the growth of total variation is precisely the reason we introduced the regularizing transfer operator $\cL$ in \eqref{eq:multi_transfer}.

Let us denote $G_p = \phi \cdot g \cdot g\circ \cF\cdots g\circ \cF^p$, where  $g=e^{itf/\sqrt{k\Var S_p}}$. In particular, $G_{p-1} = \phi w_1$.
With the aid of the identity
\beqn
\cL(f\circ \cF\cdot g) = f\cdot \cL g
\eeqn
we are able to write the recursion relation
\beqn
\cL^p G_{p}  = \cL^p(G_{p-1} \cdot g\circ \cF^p) = \cL^p(G_{p-1})\cdot g = \cL(\cL^{p-1}G_{p-1})\cdot g,
\eeqn
because $\cL^p$ is the transfer operator of $\cF^p$. Now,
\beqn
\bigvee_0^1 \cL^p G_{p}  \leq  {\norm{g}}_\infty \bigvee_0^1\cL(\cL^{p-1}G_{p-1}) + {\norm{\cL^{p}G_{p-1}}}_\infty \bigvee_0^1 g,
\eeqn
where $|g|=1$ yields $|\cL^{p}G_{p-1}|  \leq \cL^{p}|G_{p-1}|  = \cL^{p}\phi = \phi $ so that  ${\norm{\cL^{p}G_{p-1}}}_\infty\leq {\norm{\phi}}_\infty$. By \eqref{eq:pw_LY_Leb},
\beqn
\bigvee_0^1 \cL^p G_{p}  \leq 2\lambda^{-1} \bigvee_0^1\cL^{p-1}G_{p-1} + A +  {\norm{\phi}}_\infty \bigvee_0^1 g,
\eeqn
which can be iterated to prove
\beqn
\begin{split}
\bigvee_0^1 \cL^p G_{p}  & \leq (2\lambda^{-1})^p\bigvee_0^1 (\phi g) + \frac{A +  {\norm{\phi}}_\infty \bigvee_0^1 g}{1-2\lambda^{-1}}% \\
 \leq (2\lambda^{-1})^p\bigvee_0^1 \phi + \frac{A +  2{\norm{\phi}}_\infty \bigvee_0^1 g}{1-2\lambda^{-1}}
\end{split}
\eeqn
for every value of $p$. Finally, $\bigvee_0^1 g \leq  (|t|/\sqrt{k\Var S_p}) \bigvee_0^1 f = \cO(1/\sqrt{n}) \bigvee_0^1 f$ for fixed values of $t$. Hence, $\sup_n \bigvee_0^1 \cL^{p-1}(\phi w_1) < \infty$ and
\beq\label{eq:exp_mcorr}
| \average{w_1\cdots w_r} - \average{w_1}\average{w_2\cdots w_r} | \leq C\Lambda^{-q}
\eeq
for all $2\leq r\leq k$ and all $n$. This bound implies the CLT due to Corollary \ref{cor:corr_sum_CLT}.

We finish with a discussion of the map $\cF:x\mapsto 2x\mod 1$. In this case the situation is quite a bit simpler than above, and \eqref{eq:pw_LY_Leb} is not needed. First of all, $\phi\equiv 1$. Second, $\bigvee_0^1 \cL g \leq \frac12 \bigvee_0^{\frac12} g + \frac12 \bigvee_{\frac12}^1 g= \frac12 \bigvee_0^{1} g$ and $\bigvee_0^1 g\circ \cF^\ell \leq 2^\ell(\bigvee_0^1 g+|g(1)-g(0)|) \leq 2^{\ell+1} \bigvee_0^1 g$ hold for all $g\in BV$. Using these facts and the bound $\bigvee_0^1 fg \leq \bigvee_0^1 f \, {\norm{g}}_\infty  + {\norm{f}}_\infty \bigvee_0^1 g$ recursively,
\beqn
\begin{split}
 \bigvee_0^1 \cL^{p-1} w_1 & \leq \frac{1}{2^{p-1}} \bigvee_0^1 w_1 %\\
%& \leq \frac{1}{2^{p-1}} \sum_{\ell = 0}^{p-1}  {\norm{g}}_\infty \cdots {\norm{g\circ \cF^{\ell-1}}}_\infty \cdot\bigvee_0^1 g\circ \cF^{\ell} \cdot {\norm{g\circ \cF^{\ell+1}}}_\infty \cdots {\norm{g\circ \cF^{p-1}}}_\infty \\
%&
\leq \frac{1}{2^{p-1}} {\norm{g}}_\infty^{p-1} \sum_{\ell = 0}^{p-1} \bigvee_0^1 g\circ \cF^{\ell}  
%\\
%& \leq \frac{1}{2^{p-1}} {\norm{g}}_\infty^{p-1} \bigvee_0^1 g\sum_{\ell = 0}^{p-1} 2^{\ell+1}
\leq 4{\norm{g}}_\infty^{p-1} \bigvee_0^1 g.
\end{split}
\eeqn
Finally, we recall that ${\norm{g}}_\infty = 1$ and insert the bound above into \eqref{eq:expanding_multicorr}. Again, \eqref{eq:exp_mcorr} follows.

%%%%%%%%%%%%%%%%%%%%%%%%%%%%%%%%%%%%%%%%%
%%%%%%%%%%%%%%%%%%%%%%%%%%%%%%%%%%%%%%%%%

\section{Conclusion} 
We have argued that detailed information about pair correlations for suitable classes of observables may be sufficient for proving the Central Limit Theorem for a given dynamical system. We have then shown that, for Sinai Billiards in two dimensions as well as transitive Anosov diffeomorphisms in any dimension, such information can be encoded into a single pair correlation estimate. In the case of Billiards, the estimate has been obtained by relaxing the regularity assumptions of an estimate in \cite{Chernov-BilliardsCoupling}. An estimate in \cite{BressaudLiverani} works readily in the Anosov case.

Interesting in its own right is the fact that in both cases the pair correlation estimates implied good estimates on multiple correlations, which tend to be difficult to bound. As suggested in \cite{Chernov-BilliardsCoupling}, such bounds yield limit theorems more sophisticated than the CLT, including the Weak Invariance Principle and Almost Sure Invariance Principle. We have chosen not to discuss these extensions here as it would seem to produce little that is new.

Let us mention that the pair correlation bound for Billiards is the product of a coupling argument, originally due to Young \cite{Young} and further refined by Dolgopyat (see \cite{ChernovDolgopyatBBM}, \cite{Chernov-BilliardsCoupling}, and \cite{ChernovMarkarian}). The proof of the pair correlation bound for Anosov diffeomorphims in \cite{BressaudLiverani} is likewise based on a coupling method. This approach has turned out to be very flexible and adaptable to many kinds of systems with some hyperbolicity.

It is then informative to recognize that the proof of the CLT for Billiards presented above depends, at the formal level, little on the fact that the rate of correlation decay is exponential. Rather, the stability property of the classes of observables (Lemma~\ref{lem:stability}) and the form of the prefactor in the correlation estimate (Remark~\ref{rem:multicorr_bound}) are the ingredients that count. Similar remarks apply to the Anosov case. It would be interesting to know if it is possible to derive a pair correlation bound that implies a strong bound on multiple correlations and hence the CLT in a slowly mixing example. 

We have also discussed non-invertible systems in the setting of piecewise expanding interval maps and observed that the situation seems rather different. Further work is needed.

Finally, we propose the following objective, which has motivated our work, to think about: formulate checkable conditions on pair correlations for a dynamical system --- likely more complicated than a single estimate --- under which the CLT holds. Advances in this direction are not only of theoretical interest but could benefit the applied scientist.

%\appendix

%%%%%%%%%%       References      %%%%%%%%%

\end{document}